\documentclass[11pt,reqno,a4paper]{amsart}
 \usepackage{amsgen, amstext,amsbsy,amsopn, amsthm, amsfonts,amssymb,amscd,amsmath,euscript,enumerate,url,verbatim,calc,tikz}
\usepackage{hyperref}
\usepackage{MnSymbol}
\usepackage{pgfkeys}
\usepackage{mathtools}
\usepackage{etex}
 \RequirePackage{etex}
\usepackage[left=1.2in,right=1.2in,bottom=1.5in]{geometry}

\usepackage{thmtools, thm-restate}
\usepackage{pst-node}

\usepackage{mathrsfs}
\usepackage{tikz-cd}

\usepackage{tikz}
\usepackage{tikz-network}
\usepackage{eqnarray,amsmath}

\let\emptyset\varnothing

\def\multiset#1#2{\ensuremath{\left(\kern-.3em\left(\genfrac{}{}{0pt}{}{#1}{#2}\right)\kern-.3em\right)}}

\usetikzlibrary{arrows}

 \usepackage{latexsym}
 \usepackage{graphics}
 \usepackage{color}
\usepackage{lastpage}
\usepackage{fancyhdr}
\usepackage{multirow}
\allowdisplaybreaks
\usepackage{graphicx}
\graphicspath{ {F:/IMAGES/} }

\makeatletter
\def\oversortoftilde#1{\mathop{\vbox{\m@th\ialign{##\crcr\noalign{\kern3\p@}%
      \sortoftildefill\crcr\noalign{\kern3\p@\nointerlineskip}%
      $\hfil\displaystyle{#1}\hfil$\crcr}}}\limits}

\def\sortoftildefill{$\m@th \setbox\z@\hbox{$\braceld$}%
  \braceld\leaders\vrule \@height\ht\z@ \@depth\z@\hfill\braceru$}

\makeatother

  \newcommand{\Ass}{\operatorname{Ass}}

  \newcommand{\reg}{\operatorname{reg}}

 \newcommand{\depth}{\operatorname{depth}}

\newcommand{\inma}{\operatorname{im}}

\newcommand{\proset}{\,\mathrel{\lower 4pt\hbox{$\scriptscriptstyle/$}
\mkern -14mu\subseteq }\,} 

 \newtheorem{theorem}{Theorem}[section]
 \newtheorem{corollary}[theorem]{Corollary}
 \newtheorem{lemma}[theorem]{Lemma}
 \newtheorem{proposition}[theorem]{Proposition}

 \newtheorem{conjecture}[theorem]{Conjecture}

\usepackage{amsmath}

 \theoremstyle{definition}
 
 \newtheorem{remark}[theorem]{Remark}
 \newtheorem{definition}[theorem]{Definition}

 \newtheorem{example}[theorem]{Example}

\title[Asymptotic behaviour and stability index of v-numbers]{Asymptotic behaviour and stability index of v-numbers of graded ideals}
\author{Prativa Biswas, Mousumi Mandal and Kamalesh Saha }

\thanks{AMS Classification 2010: 13A02, 13F20, 13F55, 05E40, 05C38}
\thanks{Key words and phrases: v-number, graded ideals, asymptotic behaviour, linear powers, stability index}
\address{Department of Mathematics, Indian Institute of Technology Kharagpur, 721302, India}\email{prativabiswassnts@kgpian.iitkgp.ac.in}
\address{Department of Mathematics, Indian Institute of Technology Kharagpur, 721302, India}\email{mousumi@maths.iitkgp.ac.in}
\address{Chennai Mathematical Institute, Siruseri, Chennai, Tamil Nadu 603103, India}\email{ksaha@cmi.ac.in}
\begin{document}

\maketitle
\begin{abstract}
Recently, Ficarra and Sgroi initiated the study of v-numbers of powers of graded ideals. They proved that for a graded ideal $I$ in a polynomial ring $S$, $\mathrm{v}(I^k)$ is a linear function in $k$ for $k>>0$. Later, Ficarra conjectured that if $I$ is a monomial ideal with linear powers, then $\mathrm{v}(I^k)=\alpha(I)k-1$ for all $k\geq 1$, where $\alpha(I)$ denotes the initial degree of $I$. In this paper, we generalize this conjecture for graded ideals. We prove this conjecture for several classes of graded ideals: principal ideals, ideals $I$ with $\depth(S/I)=0$, cover ideals of graphs, $t$-path ideals, monomial ideals generated in degree $2$, edge ideals of weighted oriented graphs. We reduce the conjecture for several classes of graded ideals (including square-free monomial ideals) by showing it is enough to prove the conjecture for $k=1$ only. We define the stability index of the $\mathrm{v}$-number for graded ideals and investigate the stability index for edge ideals of graphs. 
\end{abstract}

\section{Introduction}
Let $S=K[x_1,\ldots,x_n]=\bigoplus_{d\geq 0}S_d$ be a standard graded polynomial ring in $n$ variables over a field $K$ and $I\subset S$ be a graded ideal. We denote the set of associated primes of $I$ by $\Ass(I).$ The minimal elements of $\Ass(I)$ with respect to inclusion are called the \textit{minimal primes} of $I$, and the rest of the elements are called the \textit{embedded primes} of $I$. The idea of the invariant in this context, namely $\mathrm v$-number, was introduced in \cite{cooper}, and further studied in \cite{ass23, biswasmandal23, civan, ficarra, sgroi, rey, jaramillo, jsvbinom23, kotalsaha23, kamalesh, saha} for graded ideals. Recently, \cite{concavnum}, \cite{luca} and \cite{ghosh} considered a more general description of this invariant. In this article, we focus on the behaviour of the $\mathrm{v}$-number of powers of graded ideals. The definition of the \textit{$\mathrm{v}$-number} of a proper graded ideal $I\subset S$, denoted by $\mathrm v(I)$, is as follows:
$$\mathrm v(I):=\min\{d\geq 0 ~|~ \exists f\in S_d\text{ and } \mathfrak{p}\in \Ass(I) ~\mbox{with}~ (I:f)=\mathfrak{p}\}.$$
    For each $\mathfrak{p}\in \Ass(I)$, the $\mathrm v$-number of $I$ at $\mathfrak{p}$, known as the \textit{local $\mathrm{v}$-number} of $I$ at $\mathfrak{p}$, is denoted by $\mathrm v_\mathfrak{p}(I)$ and defined as
  \begin{center}
      $\mathrm v_\mathfrak{p}(I):=\min\{d\geq 0 ~|~ \exists f\in S_d ~\mbox{with} ~(I:f)=\mathfrak{p}\}$.
  \end{center}
 In \cite{jaramillo}, Jaramillo and Villarreal studied the $\mathrm{v}$numberod edge ideals. It has been proved for several classes of graphs that $\mathrm{v}(I(G))\leq \reg(S/I(G))$ (see \cite{rey, jaramillo, saha}), where $I(G)$ denotes the edge ideals of a graph $G$. Later, Civan \cite{civan} proved that $\mathrm{v}(I(G))$ can be arbitrarily larger than $\reg(S/I(G))$. However, in \cite{kamalesh}, the third author proved that for the cover ideals of graphs (Alexander dual of edge ideals), the $\mathrm{v}$-number serves as a sharp lower bound for the regularity of the corresponding quotient ring. Also, in \cite{kotalsaha23}, the authors investigate when the $\mathrm{v}$-number can be greater than or equal to the regularity. Hence, people are interested in seeing how the $\mathrm{v}$-number behaves by relating it with the studies of other known invariants. One of the remarkable works in the theory of Castelnuovo-Mumford regularity is the study of the asymptotic behaviour of the regularity. In \cite{cutkosky},  Cutkosky \textit{et al.}, and in \cite{kodi} independently, Kodiyalam proved that the regularity of $I^k$ is an asymptotic linear function in $k$. Motivated by this study, Ficarra and Sgroi initiated the study of asymptotic $\mathrm{v}$-numbers of graded ideals \cite{sgroi}. Surprisingly, the so-called $\mathrm{v}$-function $\mathrm{v}(I^k)$ turns out to be an asymptotic linear one (see \cite[Theorem 3.1]{sgroi}), i.e., for $k$ large enough, $\mathrm v(I^k)$ is a linear function in $k$ of the form 
 $$\mathrm v(I^k)=\alpha(I)k+b,$$
 where $b\in\mathbb{Z}$ and $\alpha(I)=\min\{\deg(f)~|~f\in I\setminus\{0\}\}$. Therefore, we can define the \textit{stage of stabilization} or \textit{stability index} of the $\mathrm{v}$-number of $I$ as follows:
 $$\text{v-stab}(I):=\min\{t\mid \mathrm{v}(I^k)=\alpha(I)k+b,\text{ for all } k\geq t\}.$$
 In \cite{ficarra}, Ficarra conjectured that if $I\subset S$ is a monomial ideal with linear powers, then $\mathrm{v}(I^k)=\alpha(I)k-1$. If a monomial ideal $I$ has linear powers, then the regularity of $I^k$ attains its minimum value for all $k\geq 1$, i.e., $\reg(I^k)=\alpha(I^k)=\alpha(I)k$. Thus, the $\mathrm{v}$-number may attain its minimum value whenever $I$ has linear powers. This was the motivation behind Ficarra's conjecture. The author verified the conjecture for several classes of monomial ideals with linear powers in \cite{ficarra}.\par 

In this paper, we give a sharp lower bound of the $\mathrm{v}$-number for graded ideals, which helps us to generalize \cite[Conjecture 2.6]{ficarra} for arbitrary graded ideals as follows:
\medskip

\noindent \textbf{Conjecture \ref{conjlp}.} \textit{Let $I\subset S$ be a graded ideal with linear powers. Then 
$$\mathrm{v}(I^k)=\alpha(I)k-c(I)~\text{for all}~k\geq 1,$$
where $c(I):=\max\{\alpha(\mathfrak{p})\mid \mathfrak{p}\in\Ass(I)\}$.}
\medskip

\noindent This conjecture is one of the most important parts of this paper. We investigate when the lower bound can be attained. In several cases, we have shown that if $\mathrm{v}(I)$ attains its lower bound, then $\mathrm{v}(I^k)$ does so. We prove Conjecture \ref{conjlp} for several classes of graded ideals. Also, for some classes (including square-free monomial ideals), we reduce the conjecture in the sense that it is enough to prove the conjecture for $k=1$. Furthermore, we investigate some sufficient conditions for which $\mathrm v(I^k)$ is a linear function from the very beginning. In the last section of this paper, we study the asymptotic behaviour of the $\mathrm{v}$-number of edge ideals of graphs. We explicitly find the linear function $\mathrm{v}(I(G)^k)$ for $k>>0$. Also, we investigate the stability index of $\mathrm{v}$-number for edge ideals, i.e., investigate the integer $k_0$ such that $\mathrm{v}(I(G)^k)$ is a linear function in $k$ for all $k\geq k_0$. The paper is organized in the following manner.\par 

Section \ref{secpreli} of this paper is on some prerequisite definitions, notions and results for a better understanding of the subsequent sections. In Section \ref{secgradedconjonv}, we first derive a tight lower bound of the $\mathrm{v}$-number for graded ideals (Lemma \ref{low}). Then, we show that if $I$ is a normally torsion-free graded ideal without any embedded prime or $I$ is an equigenerated graded ideal having strong persistence property, then the $\mathrm{v}$-number of $I^k$ attains the lower bound given in Lemma \ref{low} whenever $\mathrm{v}(I)$ does so (Proposition \ref{propntf}, \ref{propsppv}). Next, we generalize Ficarra's conjecture (\cite[Conjecture 2.6]{ficarra}) from monomial ideals to graded ideals (see Conjecture \ref{conjlp}). In Theorem \ref{thmprincipal} and \ref{thmdepth0}, we prove Conjecture \ref{conjlp} for two classes of graded ideals, which may not be monomial ideals. Section \ref{secmonv} of this paper is devoted to studying the $\mathrm{v}$-number of powers of certain classes of monomial ideals and finding some monomial class in support of Conjecture \ref{conjlp}. First, we reduce Conjecture \ref{conjlp} for square-free monomial ideals in Theorem \ref{thmsqfree}, which says for a square-free monomial ideal with linear powers, it is enough to prove the conjecture for $k=1$ only. As a corollary of Theorem \ref{thmsqfree}, we prove Conjecture \ref{conjlp} for cover ideals of graphs with linear powers (Corollary \ref{corcover}), $t$-path ideals with linear powers (Corollary \ref{cort-path}). We generalize Theorem \ref{thmsqfree} for monomial ideals with some conditions. In Proposition \ref{propv=1}, we show that if $I$ is a monomial ideal with $\mathrm{v}(I)=1$, then $\mathrm{v}(I^k)=\alpha(I)k-1=2k-1$ for all $k\geq 1$. This may not hold for monomial ideal $I$ with $\mathrm{v}(I)>1$ or $\mathrm{v}(I)\alpha(I)-1$ (see Example \ref{exampowerneqlb}). Finally, we prove Conjecture \ref{conjlp} for two more classes of monomial ideals (need not be square-free), such as monomial ideal generated in degree $2$ having linear resolution (Theorem \ref{thmdeg2mon}) and edge ideals of weighted oriented graphs having linear resolution (Theorem \ref{thmwoglp}). We end this section by computing the exact value of $\mathrm{v}$-numbers of powers of vertex splittable ideals (see Theorem \ref{thmvsplit}). In Section \ref{secvpoweredge}, we investigate the stability index of $\mathrm{v}$-number for edge ideals of graphs. We show in Theorem \ref{bound} that for a connected graph $G$ with $m$ edges, $\mathrm{v}(I(G)^k)=2k-1$ for all $k\geq m+1$, i.e., $\text{v-stab}(I(G))\leq m+1$. However, for disconnected graphs, the linear function $\mathrm{v}(I(G))^k$ for large $k$ may not be equal to $2k-1$ (see Example \ref{examdiscon}). Thus, we may expect that for an equigenerated square-free monomial ideal $I$, which can not be written as a sum of two square-free monomial ideals in two polynomial rings with disjoint variables, that $v(I^k)=\alpha(I)k-1$ for $k>>0$. But this fact also may not be true and is clear from Example \ref{examsqfreenottrue}. We give a better upper bound (sharp) of $\text{v-stab}(I(G))$ for any bipartite graph $G$ (Theorem \ref{thmvstabbipartite}). As a corollary, we get $\mathrm{v}(I(G)^k)=2k-1$ for all $k\geq [\frac{n}{2}]-1$ if $G$ is a path (Corollary \ref{corvstabpath}) or even cycle (Corollary \ref{corvstabevencycle}). For odd cycles, we prove that $\mathrm{v}(I(G)^k)=2k-1$ for all $k\geq [\frac{n}{2}]$ (Proposition \ref{propvstaboddcycle}).

\section{Preliminaries}\label{secpreli}

In this section, we recall preliminary notions and terminologies that will be used throughout the article.
\par

In $S$, an element is said to be a \textit{monomial} if it is of the form $x_{1}^{a_1}\ldots x_{t}^{a_t}$, where $(a_1,\ldots,a_t)\in \mathbb{Z}_{\geq 0}^{t}$ and $\mathbb{Z}_{\geq 0}$ is the set of all non-negative integers. An ideal $I\subset S$ is said to be a \textit{monomial ideal} if it is minimally generated by a set of monomials in $S$. For a monomial ideal $I$, the set of minimal generators of $I$ is unique, and we denote it by $\mathcal{G}(I)$. 

\begin{definition}
    Let $I,J\subset S$ be two ideals. Then $(I:J):=\{f\in S\mid fg\in I~\mbox{for all}~g\in J\}$ is an ideal of $S$, known as the \textit{colon ideal} of $I$ with respect to $J$. For $f\in S$, we write $(I:f):=(I:\langle f\rangle)$. By \cite[ Proposition 1.2.2]{Hibi}, for a monomial ideal $I\subset S$ and a monomial $f\in S$, we have
    $$(I:f)=\big\langle \frac{u}{\mathrm{gcd}(u,f)}\mid u\in\mathcal{G}(I)\big\rangle.$$
\end{definition}

\begin{definition}
    A simple graph ${G}$ is defined by a pair $(V({G}), E({G}))$, where $V({G})$ is a finite set, called \textit{vertex set} of $G$ and  $E({G})$, called \textit{edge set} of $G$, is the family of subsets of $V(G)$ such that they are pairwise incomparable with respect to inclusion and the cardinality of each element of $E({G})$ is two. 
\end{definition}
Suppose ${G}$ is a graph with $V({G})=\{x_1,x_2,\ldots,x_n\}$ and $E({G})$ as edge set. We consider each vertex $x_i$ as a variable of the polynomial ring $S=K[x_1,x_2,\ldots,x_n]$ in $n$ variable over a field $K$, and by defining the following ideal, we can relate each graph to an ideal.

The edge ideal associated with a graph $G$ is the monomial ideal
    $$I({G}):=\langle\prod_{x_i\in e}x_i~|~e\in E({G})\rangle.$$
A subset $C \subset V({G})$ is called a vertex cover of ${G}$ if $C \cap e 	\neq \emptyset$ for all $e \in E({G})$. If a
vertex cover is minimal with respect to inclusion, then we call it a minimal vertex cover.
\\Another associated monomial ideal with graph is the cover ideal, denoted by $J({G})$ and defined as follows:
$$J({G})=\langle \displaystyle\prod_{x_i\in C}x_i~|~C~\mbox{is a minimal vertex cover of }{G}\rangle.$$

 These edge ideals and cover ideals are square-free monomial ideals. Hence we can talk about $\mathrm v$-number of edge ideal. 
Next, we recall Some definitions related to graphs that we have considered in this paper:
 \begin{definition}

 For a vertex $x$ in $V(G)$, the neighbor set of $x$ in ${G}$, denoted by $N_{{G}}(x)$, is defined by
    $$N_{{G}}(x)=\{x_i\in V({G})~|~\{x_i,x\} ~\mbox{is}~ \text{an edge}~ of~{G}\}.$$
The degree of $x$, denoted by $\deg (x)$, is 
$\deg (x)=|N_{{G}}(x)|.$
    
 \end{definition}
\begin{definition}

    If $M$ is a set of pairwise disjoint edges of a graph, then $M$ is called a matching. An induced matching of a graph ${G}$ is a matching $M=\{e_1,e_2,\ldots,e_m\}$ of ${G}$ such that the only edges of ${G}$ contained in $\displaystyle\bigcup_{i=1}^{m}{e_i}$
 are $e_1,e_2,\ldots,e_m$.
\\  The induced matching number of ${G}$ is the number of edges in the largest induced matching and is denoted by $\mbox{im}({G})$.
\end{definition}
\begin{definition}
    A simple graph is said to be a path if its vertices can be ordered such that two vertices are adjacent if and only if they are consecutive in the list. If a path has $n$ vertices, then it is denoted by $P_n$.
\end{definition}
\begin{definition}
    A cycle is a simple graph with the same number of vertices and edges whose vertices can be placed around a circle so that two vertices are adjacent if and only if they appear consecutively along the circle. A cycle with $n$ vertices is denoted by $C_n$.
\end{definition}
\begin{definition}
    a simple graph $G$ is bipartite if and only if its vertex set $V(G)$ can be partitioned into two non-empty subsets $X$ and $Y$, such that every edge in $E(G)$ has one endpoint in $X$ and the other endpoint in $Y$.
\end{definition}
\begin{definition}[\cite{ban}]
    Two vertices $u$ and $ v$ ($u$ may be same as $ v$) are said to be even connected with respect to an $s$-fold product $e_1e_2\ldots e_s$ if there exists a path $x_1x_2\ldots x_{2k+2}$, $k\geq2$ in graph $G$ such that:
    \begin{enumerate}[(i)]
    \item $x_1=u,~x_{2k+2}=v.$
    \item For all $1\leq l\leq k$, $x_{2l}x_{2l+1}=e_i$ for some $e_i$.
    \item For all $i$,
    $$\{l\geq 1~|~x_{2l}x_{2l+1}=e_i\}\leq |\{j\mid e_j=e_i\}|.$$
    \item For all $1\leq r\leq 2k+1$, $x_rx_{r+1}$ is an edge in $G$.
    \end{enumerate}
    
\end{definition}
In similar manner, we can define odd connected vertices.
Now the below-mentioned theorem gives the idea of how to generate even-connected edges.
\begin{theorem}\cite[Theorem 6.7.]{ban}\label{even}
Every generator $uv$ $(u$ may be equal to $v)$ of $(I(G)^{s+1}:e_1e_2\ldots e_s)$ is either an edge of $G$ or even-connected with respect to $e_1\ldots e_s,~s\geq1.$
    \end{theorem}

For a non-zero graded ideal $I$, define $\alpha (I):=\min\{\deg (f)~|~f\in I\setminus \{0\}\}$.

Next, we define another important invariant of commutative algebra.
    Let $M$ be a finitely generated $S$ module. We can write the graded minimal free
resolution of $M$ in the following form
    $$0\rightarrow\displaystyle\bigoplus_{j} S(-j)^{\beta_{g,j}}\rightarrow\cdot\cdot\cdot\rightarrow\displaystyle\bigoplus_j S(-j)^{\beta_{1,j}}\rightarrow \displaystyle\bigoplus S(-j)^{\beta_{0,j}}\rightarrow M\rightarrow 0,$$
    where $\beta_{ij}$ is the $(i,j)$-th graded Betti number of $M$ and $S(-j)$ denote the polynomial ring shifted in degree $j$.
\begin{definition}
    The Castlenuovo-Mumford regularity (or regularity in short) of $M$, denoted as $\reg(M)$, is defined by 
    $$\reg(M)=\max\{j-i~|~\beta_{i,j}\neq 0\}.$$
\end{definition}
Another important term related to regularity is linear resolution. An equigenerated graded ideal $I$ has linear resolution, if and only if, $\reg(I)=\alpha(I).$ It is said that $I$ has linear power if $I^k$ has linear resolution, for all $k\geq1.$

\section{Generalization of Ficarra's conjecture}\label{secgradedconjonv}

In this section, we derive a general lower bound of $\mathrm{v}$-number for graded ideals and generalize Ficarra's conjecture on the $\mathrm{v}$-numbers of powers of monomial ideals having linear powers for graded ideals. We give some sufficient conditions for $\text{v-stab}(I)$ to be $1$, and the linear function would be the possible lower bound of $\mathrm{v}(I^k)$. We provide certain classes of graded ideals (need not be monomial) in support of our conjecture.\par 

The following lemma is the generalization of \cite[Proposition 2.2]{ficarra}.

\begin{lemma}\label{low}
    Let $I\subset S$ be a graded ideal. Then for every $\mathfrak{p}\in\Ass(I)$, we have
    $\mathrm{v}_{\mathfrak{p}}(I)\geq\alpha(I)-\alpha(\mathfrak{p})$, and hence, $\mathrm{v}(I)\geq \alpha(I)-c(I)$. Moreover, if all the associated primes of $I$ are generated by linear forms (for example, $I$ is a monomial ideal), then $$\mathrm{v}_{\mathfrak{p}}(I)\geq \mathrm{v}(I)\geq \alpha(I)-1 \text{ for all } \mathfrak{p}\in\Ass(I).$$
\end{lemma}

\begin{proof}
Let $f\in S$ be a homogeneous polynomial such that $(I:f)=\mathfrak{p}$ for some $\mathfrak{p} \in \Ass(I)$ and $\deg (f)=\mathrm{v}_{\mathfrak{p}}(I)$. Now, choose an element $a\in\mathfrak{p}$ such that $\deg(a)=\alpha(\mathfrak{p})$. Then we have $fa\in I$, which implies $\deg(fa)\geq \alpha(I)$. Therefore, $\deg(f)\geq \alpha(I)-\deg(a)=\alpha(I)-\alpha(\mathfrak{p})$. Thus, $\mathrm{v}_{\mathfrak{p}}(I)\geq \alpha(I)-\alpha(\mathfrak{p})$. Suppose $\mathrm{v}(I)=\mathrm{v}_{\mathfrak{p}}(I)$ for some $\mathfrak{p}\in\Ass(I)$. Then $\mathrm{v}(I)\geq \alpha(I)-\alpha(\mathfrak{p})\geq \alpha(I)-c(I)$.\par
If all the associated primes of $I$ are generated by linear forms, then $\alpha(\mathfrak{p})=1$ for all $\mathfrak{p}\in\Ass(I)$. Hence, the result follows.
\end{proof}

\begin{definition}\label{defntf}{\rm
    A proper ideal $I$ of a commutative Noetherian ring $A$ is said to be \textit{normally torsion-free} if $\Ass(I^k)$ is contained in $\Ass(I)$ for all $k\geq 1$.
}
\end{definition}

If $I$ is a normally torsion-free ideal with no embedded prime, then it follows from Definition \ref{defntf} that $\Ass(I^k)=\Ass(I)$ for all $k\geq 1$. Again, if $I$ has no embedded prime, then $I$ is normally torsion-free if and only if $I^k=I^{(k)}$ for all $k\geq 1$, where $I^{(k)}$ denote the $k$-th symbolic power of $I$ (see \cite[Proposition 4.3.29]{vil15}).

\begin{proposition}\label{propntf}
    Let $I\subset S$ be a graded ideal having no embedded prime. Suppose $I$ is normally torsion-free (equivalently, $I^k=I^{(k)}$ for all $k\geq 1$). If $\mathrm v_{\mathfrak{p}}(I)=\alpha(I)-\alpha(\mathfrak{p})$ for some $\mathfrak{p}\in \Ass(I)$, then $\mathrm v_{\mathfrak{p}}(I^k)=\alpha(I)k-\alpha(\mathfrak{p})$ for all $k\geq 1$. Moreover, if $\mathrm{v}(I)=\alpha(I)-c(I)$, then $\mathrm{v}(I^k)=\alpha(I)k-c(I)$ for all $k\geq 1$.
\end{proposition}
\begin{proof}
    Let $\mathrm{v}_{\mathfrak{p}}(I)=\alpha(I)-\alpha(\mathfrak{p})$ for some $\mathfrak{p}\in\Ass(I)$. Then there exists a homogeneous polynomial $f\in S$ such that $(I:f)=\mathfrak{p}$ and $\deg (f)=\mathrm{v}_{\mathfrak{p}}(I)$. Now, let us consider $g=fa$ for some homogeneous element $a\in\mathfrak{p}$ with $\deg (a)=\alpha(\mathfrak{p})$. Then $g\in I$, which imply $\mathfrak{p}\subset (I^k:g^{k-1}f)$ as $\mathfrak{p}\subset (I:f)$. By the given condition, we have $\Ass(I^k)=\Ass(I)$. Since $\Ass(I^k:g^{k-1}f)\subset\Ass(I^k)$ and $I^k$ has no embedded prime, we have $(I^k:g^{k-1}f)=\mathfrak{p}$. This gives $\mathrm{v}_{\mathfrak{p}}(I^k)\leq \deg(g^{k-1}f)=\alpha(I)k-\alpha(\mathfrak{p})$. Therefore, using Lemma \ref{low}, we get $\mathrm v_{\mathfrak{p}}(I^k)=\alpha(I)k-\alpha(\mathfrak{p})$.\par
    
    Suppose $\mathrm{v}(I)=\mathrm{v}_{\mathfrak{p}}(I)=\alpha(I)-c(I)$ for some $\mathfrak{p}\in\Ass(I)$. Note that $\alpha(I)-c(I)\leq \alpha(I)-\alpha(\mathfrak{p})$. Again, Lemma \ref{low} implies that $\mathrm{v}_{\mathfrak{p}}(I)=\alpha(I)-c(I)\geq \alpha(I)-\alpha(\mathfrak{p})$. Thus, we have $c(I)=\alpha(\mathfrak{p})$, and hence, it follows from the first part that $\mathrm{v}_{\mathfrak{p}}(I^k)=\alpha(I)k-c(I)$. Therefore, by the definition of $\mathrm{v}$-number and Lemma \ref{low}, the result follows.
\end{proof}

\begin{proposition}\label{propsppv}
    Let $I\subset S$ be an equigenerated graded ideal with strong persistence property (equivalently, $(I^{k+1}:I)=I^k$ for all $k\geq 1$). If $\mathrm{v}_{\mathfrak{p}}(I)=\alpha(I)-\alpha(\mathfrak{p})$ for some $\mathfrak{p}\in\Ass(I)$, then $\mathrm{v}_{\mathfrak{p}}(I^k)=\alpha(I)k-\alpha(\mathfrak{p})$ for all $k\geq1$. In particular, if $\mathrm{v}(I)=\alpha(I)-c(I)$, then $\mathrm{v}(I^k)=\alpha(I)k-c(I)$ for all $k\geq 1$.
\end{proposition}
\begin{proof}
    We may assume $k\geq 2$. Let $f$ be a homogeneous element such that $(I:f)=\mathfrak{p}$ for some $\mathfrak{p}\in \Ass (I)$ and $\deg (f)=\mathrm v_{\mathfrak{p}}(I)=\alpha(I)-\alpha(\mathfrak{p})$. Since $I$ is equigenerated graded ideal, $I^{k-1}$ is so. Let $\{f_1,\ldots, f_r\}$ be a minimal homogenous equigenerated system of generators of $I^{k-1}$. By the given condition, we have 
    \begin{align*}
        \mathfrak{p}=(I:f)&=((I^{k}:I^{k-1}):f)
        \\&=(I^{k}:fI^{k-1})\\&=(I^{k}:\langle ff_1,\ldots,ff_r\rangle)
        \\&=\displaystyle\bigcap_{i=1}^{r}(I^{k}:ff_i).
    \end{align*}
    Therefore, $(I^k:ff_i)=\mathfrak{p}$ for some $1\leq i\leq r$. Since $\deg(f_i)=(k-1)\alpha(I)$, by the definition of $\mathrm{v}$-number, it follows that $\mathrm v_{\mathfrak{p}}(I^{k})\leq \deg(ff_i)=\alpha(I)k-\alpha(\mathfrak{p})$ for all $k\geq 2$. Applying Lemma \ref{low}, we get $\mathrm v_{\mathfrak{p}}(I^{k})=\alpha(I)k-\alpha(\mathfrak{p})$ for all $k\geq 1$.\par 

    Now, suppose $\mathrm{v}(I)=\alpha(I)-c(I)$. Let us choose $\mathfrak{p}\in\Ass(I)$ for which $\mathrm{v}(I)=\mathrm{v}_{\mathfrak{p}}(I)$. Note that $\mathrm{v}_{\mathfrak{p}}(I)=\alpha(I)-c(I)\leq \alpha(I)-\alpha(\mathfrak{p})$. Again, Lemma \ref{low} implies that $\mathrm{v}_{\mathfrak{p}}(I)\geq \alpha(I)-\alpha(\mathfrak{p})$. Thus, we have $c(I)=\alpha(\mathfrak{p})$, and hence, it follows from the first part that $\mathrm{v}_{\mathfrak{p}}(I^k)=\alpha(I)k-c(I)$. Therefore, by the definition of $\mathrm{v}$-number and Lemma \ref{low}, we get $\mathrm{v}(I^k)=\alpha(I)k-c(I)$ for all $k\geq 1$.
\end{proof}

Ficarra conjectured \cite[Conjecture 2.6]{ficarra} that if $I\subset S$ is a monomial ideal with linear powers, then $\mathrm{v}(I^k)=\alpha(I)k-1$. If a monomial ideal $I$ has linear powers, then the regularity of $I^k$ attains its minimum value for all $k\geq 1$, i.e., $\reg(I^k)=\alpha(I^k)=\alpha(I)k$. Thus, the $\mathrm{v}$-number may attain its minimum value whenever $I$ has linear powers. This was the motivation behind Ficarra's conjecture. Also, he showed that this conjecture is false when $I$ is not monomial. However, he actually provided a non-monomial graded ideal $I$ with linear powers for which $\mathrm{v}(I^k)\neq \alpha(I)k-1$. Due to our Lemma \ref{low}, we can observe that in \cite[Example 2.4]{ficarra}, the possible minimum value of $I^k$ is $\alpha(I)k-2$, which is attained. Hence, we generalize \cite[Conjecture 2.6]{ficarra} as follows:

\begin{conjecture}\label{conjlp}
Let $I\subset S$ be a graded ideal with linear powers. Then 
$$\mathrm{v}(I^k)=\alpha(I)k-c(I)~\text{for all}~k\geq 1,$$
where $c(I):=\max\{\alpha(\mathfrak{p})\mid \mathfrak{p}\in\Ass(I)\}$.
\end{conjecture}

For the class of monomial ideals, Conjecture \ref{conjlp} is equivalent to \cite[Conjecture 2.6]{ficarra}. In \cite{ficarra}, the conjecture has been proved for some classes of monomial ideals: monomial ideals with linear powers having $\mathrm{depth}(S/I)=0$, edge ideals of graphs having linear powers, polymatroidal ideals, and Hibi ideals. In this paper, we prove Conjecture \ref{conjlp} for several classes of ideals, most of them are monomial ideals. First, we start by giving the following class of graded ideals, which may not be monomial, in support of our conjecture.

\begin{theorem}\label{thmprincipal}
    Let $I\subset S$ be a graded principal ideal. Then $I$ has linear powers and $\mathrm{v}_{\mathfrak{p}}(I^k)=\alpha(I)k-\alpha(\mathfrak{p})$ for all $\mathfrak{p}\in\Ass(I^k)$ and for all $k\geq 1$. In particular, we have
    $$\mathrm{v}(I^k)=\alpha(I)k-c(I)~\text{for all}~k\geq 1.$$
\end{theorem}

\begin{proof}
    Let $I=\langle f \rangle$, where $f$ is a homogeneous polynomial in $S$. We can write $f=f_{1}^{a_1}\ldots f_{r}^{a_r}$ for some irreducible polynomials $f_1,\ldots,f_r$ in $S$, where $f_i\neq f_j$ for $i\neq j$ and $a_1,\ldots, a_r$ are positive integers. Then $\mathfrak{p}_{i}:=\langle f_i\rangle$ is a prime ideal for each $1\leq i\leq r$. Note that $I=\mathfrak{p}_{1}^{a_1}\cdots \mathfrak{p}_{r}^{a_r}=\mathfrak{p}_{1}^{a_1}\cap \cdots \cap \mathfrak{p}_{r}^{a_r}$ is the primary decomposition of $I$, and thus, $\Ass(I)=\{\mathfrak{p}_{1},\ldots,\mathfrak{p}_{r}\}$. It is clear that $I$ has no embedded prime, and hence, $(I:\frac{f}{f_i})=\mathfrak{p}_{i}$ as $\frac{f}{f_i}\not\in I$. Therefore, $\mathrm{v}_{\mathfrak{p}_{i}}(I)\leq \deg(\frac{f}{f_i})=\deg(f)-\deg(f_i)=\alpha(I)-\alpha(\mathfrak{p}_{i})$. By Lemma \ref{low}, we get $\mathrm{v}_{\mathfrak{p}_i}(I)=\alpha(I)-\alpha(\mathfrak{p}_i)$. It is a routine check that $(I^{k+1}:I)=I^k$ for all $k\geq 1$. Hence, it follows from Proposition \ref{propsppv} that $\mathrm{v}_{\mathfrak{p}_i}(I^k)=\alpha(I)k-\alpha(\mathfrak{p}_i)$. Since $I$ is a principal ideal, $\Ass(I^k)=\Ass(I)$ for all $k\geq 1$, and the result follows.\par 

    By the definition of $\mathrm{v}$-number and the first part, $\mathrm{v}(I^k)=\alpha(I)k-c(I)~\text{for all}~k\geq 1$.
\end{proof}

\begin{theorem}\label{thmdepth0}
    Let $I\subset S$ be a graded ideal whose all associated primes are generated by linear forms. If $I$ has linear powers and $\depth(S/I)=0$, then 
    $$\mathrm{v}(I^k)=\alpha(I)k-1~\text{for all}~k\geq 1.$$
\end{theorem}

\begin{proof}
    Since $I$ has linear powers, $\depth(S/I^k)$ is a non-increasing function of $k$ by \cite[Proposition 10.3.4]{Hibi}. Therefore, $\depth(S/I^k)=0$ for all $k\geq 1$ as $\depth(S/I)=0$ is given. Now, $I^k$ has a linear resolution for all $k\geq 1$ imply $\reg(S/I^k)=\alpha(I^k)-1=\alpha(I)k-1$ for all $k\geq 1$. Hence, by \cite[Proposition 1.6]{ghosh}, we have $\mathrm{v}(I^k)\leq \reg(S/I^k)=\alpha(I)k-1$ for all $k\geq 1$. Let $\mathfrak{p}\in \Ass(I^k)$. Then either $\mathfrak{p}$ is a minimal prime of $I$ or $\mathfrak{p}$ contains a minimal prime of $I$. This gives $\alpha(\mathfrak{p})=1$ for all $\mathfrak{p}\in\Ass(I^k)$, and thus, $c(I^k)=1$ for all $k\geq 1$. Hence, using Lemma \ref{low}, we get $\mathrm{v}(I^k)=\alpha(I)k-1$ for all $k\geq 1$.
\end{proof}

\begin{corollary}
    Let $I\subset S$ be a $\mathfrak{m}$-primary ideal, where $\mathfrak{m}$ is the unique homogeneous maximal ideal of $S$. If $I$ has linear powers, then $\mathrm{v}(I^k)=\alpha(I)k-1$ for all $k\geq 1$.
\end{corollary}
\begin{proof}
    Since $I$ is $\mathfrak{m}$-primary, $\depth(S/I)=0$. Thus, the result follows by Theorem \ref{thmdepth0}.
\end{proof}

\section{The $\mathrm{v}$-number of powers of monomial ideals}\label{secmonv}

In this section, we reduce Conjecture \ref{conjlp} for several classes of monomial ideals. We provide some classes of monomial ideals in support of Conjecture \ref{conjlp}. Also, we explicitly find the $\mathrm{v}$-numbers of powers of vertex splittable ideals.
  
\begin{theorem}\label{thmsqfree}
    Let $I\subset S$ be a square-free monomial ideal and $\mathrm v(I)=\alpha(I)-1$. Then $\mathrm v(I^{k})=\alpha(I)k-1$ for all $k\geq 1$. 
\end{theorem}
\begin{proof}
    Since $I$ is a square-free monomial ideal and $\mathrm v(I)=\alpha(I)-1$, there exists a square-free monomial $f\in S$ with $\deg(f)=\alpha(I)-1$ such that $(I:f)=\mathfrak{p}$ for some $\mathfrak{p}\in \Ass(I)$. Without loss of generality, let us assume $\mathfrak{p}=\langle x_1,\ldots, x_t\rangle$. Then, it is clear that $x_i\nmid f$ for each $1\leq i\leq t$. Let us consider $g=fx_1\in I$. Then $x_1^{k-1}~|~g^{k-1}f$, but $x_1^{k}\nmid g^{k-1}f$. Since $(I:f)=\mathfrak{p}$, $x_ig^{k-1}f=g^{k-1}fx_i\in I^k$ for all $1\leq i\leq t$. This implies that $\mathfrak{p}\subseteq (I^k:g^{k-1}f)$. To show the reverse inequality, let $h=\frac{h_1\ldots h_k}{\mathrm{gcd}(h_1\ldots h_k,~g^{k-1}f)}\in \mathcal{G}(I^k:g^{k-1}f)$, where $h_1,\ldots, h_k\in \mathcal{G}(I)$. Since $\{x_1,\ldots,x_t\}$ is a minimal prime of $I$, for each $h_i$, there exists $j_{i}\in\{1,\ldots,t\}$ such that $x_{j_{i}}\mid h_i$. Therefore, $\prod_{i=1}^{k}h_{i}$ is divisible by $\prod_{i=1}^{t}x_{i}^{c_i}$, where $\sum_{i=1}^{t}c_i=k$. Now, we distinguish the following two possible cases:
    \begin{enumerate}[(i)]
    \item Suppose $c_1=k$. Then $c_2=\ldots=c_t=0$. In this situation, observe that $x_{1}^{k-1}\mid \mathrm{gcd}(h_1\ldots h_k,~g^{k-1}f)$, but $x_{1}^{k}\nmid \mathrm{gcd}(h_1\ldots h_k,~g^{k-1}f)$. Therefore, $h\in\langle x_1\rangle\subseteq \mathfrak{p}$.
    \item If $c_1<k$, then $c_i\neq 0$ for some $2\leq i\leq t$. Since $x_i\nmid g^{k-1}f=f^{k}x_{1}^{k-1}$ for all $2\leq i\leq t$, $x_i\nmid \mathrm{gcd}(h_1\ldots h_k,~g^{k-1}f)$. Thus, $h\in \langle x_i\rangle\subseteq \mathfrak{p}$, where $c_i\neq 0$ for some $2\leq i\leq t$.
    \end{enumerate}
    Hence, $(I^k:g^{k-1}f)=\mathfrak{p}$, which gives $\mathrm v(I^k)\leq \deg(g^{k-1}f)=k\alpha(I)-1$. By Lemma \ref{low}, it follows that $ k\alpha(I)-1\leq \mathrm v(I^k)$. Therefore, $\mathrm v(I^k)=\alpha(I)k-1$.
\end{proof}

In general, $\mathrm v(I^k)$, where $I$ is a square-free monomial ideal, is not linear for all $k\geq 1$, but if $\mathrm v(I)=\alpha(I)-1$, then $\mathrm v(I^k)$ is linear from very beginning.

\begin{corollary}\label{corcover}
    If the cover ideal $J(G)$ of a graph $G$ has a linear resolution (equivalently, $I(G)$ is Cohen-Macaulay), then $\mathrm v(J(G)^k)=\alpha(J(G))k-1$.
\end{corollary}
\begin{proof}
    From \cite[Corollary 3.9]{kamalesh}, it follows that if $J(G)$ has a linear resolution, then $\mathrm v(J(G))=\alpha(J(G))-1$. Since $J(G)$ is a square-free monomial ideal, Theorem \ref{thmsqfree} gives $\mathrm v(J(G)^k)=\alpha(J(G))k-1.$
\end{proof}

\begin{definition}
   The $t$-path ideal
$I_t(G)$ associated to $G$ is the square-free monomial ideal
$$I_t(G)=\langle x_{i_1}\ldots x_{i_t}~|~\{x_{i_1},\ldots,x_{i_t}\}~\mbox{is a $t$-path of G}\rangle$$

\end{definition}

\begin{corollary}\label{cort-path}
    Let $I\subset S$ be a $t$-path ideal of $P_n$. If $I$ has a linear resolution (equivalently, $I$ has linear powers), then $\mathrm v(I^k)=\alpha(I)k-1$ for all $k\geq1$. 
\end{corollary}
\begin{proof}
    Since $I=I_{t}(P_n)$, we have $I=\langle x_1x_2\ldots x_t, x_2\ldots x_{t+1},\ldots,x_{n-t+1}\ldots x_n\rangle$. Then it is clear that $(I:x_2\ldots x_t)=\langle x_1,x_{t+1}\rangle$ for $t\leq n\leq 2t$. Thus, $\mathrm v(I)\leq t-1=\alpha(I)-1.$ By Lemma \ref{low}, $\mathrm v(I)=\alpha(I)-1$. Therefore, using Theorem \ref{thmsqfree}, we get $\mathrm v(I^k)=\alpha(I)k-1$ for all $k\geq 1$ and $t\leq n\leq 2t$. Hence, the proof follows from the fact that $I$ has a linear resolution if and only if $t\leq n\leq 2t$ by \cite[Theorem 2.2]{shan}.
\end{proof}
\begin{remark}
    Let $I=I_{t}(P_n)$. It is known by \cite[Theorem 2.2]{shan} that $I^k$ has a linear resolution for some $k\geq 1$ $\Longleftrightarrow$ $I$ has linear powers $\Longleftrightarrow$ $t\leq n\leq 2t$ $\Longleftrightarrow$ $I^k$ has linear quotient for some $k\geq 1$ $\Longleftrightarrow$ $I^k$ has linear quotient for all $k\geq 1$. Due to Corollary \ref{cort-path}, this gives a class of monomial ideals, which satisfies \cite[Conjecture 2.6]{ficarra}.
\end{remark}

From Theorem \ref{thmsqfree}, we get that $\mathrm{v}(I^k)$ is a linear function on $k$ from the very beginning whenever $I$ is a square-free monomial ideal with $\mathrm{v}(I)=\alpha(I)-1)$. In the next theorem, we show that this fact is true for a more general situation. Since the techniques of the two proofs are different, we keep both proofs separately.

\begin{theorem}\label{thmpowermon}
    Let $I\subset S$ be a monomial ideal with $\mathrm v(I)=\alpha(I)-1$. If there exists an associated prime $\mathfrak{p}$ of $I$ and a monomial $f\in S$ such that $(I:f)=\mathfrak{p}$, $\deg(f)=\mathrm{v}(I)$, and $x_i\nmid f$ for all $x_i\in\mathcal{G}(\mathfrak{p})$, then $\mathfrak{p}\in\Ass(I^k)$ and 
    $$\mathrm v(I^{k})=\mathrm{v}_{\mathfrak{p}}(I^{k})=\alpha(I)k-1~\mbox{for all}~k\geq 1.$$
\end{theorem}
\begin{proof}
    Fix any $x_i\in\mathcal{G}(\mathfrak{p})$, and consider the monomial $g=x_if$. Then $g\in I$, and thus, it is clear that $\mathfrak{p}\subset(I^{k}:g^{k-1}f)$ for all $k\geq 1$. We claim that the reverse inclusion will also hold. To prove the reverse inclusion, we use induction on the powers of the ideal $I$. If $k=1$, then the statement follows by the given hypothesis. For some $k\geq 2$, let $(I^{k-1}:g^{k-2}f)\subset \mathfrak{p}$. We want to show that $(I^{k}:g^{k-1}f)\subset \mathfrak{p}$. Now, consider an element $h\in \mathcal{G}(I^{k}:g^{k-1}f)$. Then we can write 
    $$h=\frac{g_1g_2\ldots g_{k}}{\mathrm{gcd}(g_1g_2\ldots g_{k},g^{k-1}f)}=\frac{g_1g_2\ldots g_{k}}{\mathrm{gcd}(g_1g_2\ldots g_{k},x_{i}^{k-1}f^{k})},$$ 
    for some $g_1,\ldots,g_k\in \mathcal{G}(I)$. If any of $g_j$'s is equal to $g$, then it is easy to verify that $h\in (I^{k-1}:g^{k-2}f)\subset \mathfrak{p}$. Now, suppose $g_j\neq g$ for all $j\in\{1,\ldots,k\}$. Since $g_j\in I\subset \mathfrak{p}$, there exists $x_{i_{j}}\in\mathcal{G}(\mathfrak{p})$ such that $x_{i_j}|g_j$ for $j=1,\ldots,k$. This gives 
    $$\frac{g_1g_2\ldots g_{k}}{\mathrm{gcd}(g_1g_2\ldots g_{k},f^{k})}=x_{i_1}\ldots x_{i_k}m,$$ for some monomial $m$ as $x_{i_j}\nmid f$ for all $1\leq j\leq k$. Since $x_i\nmid f$, we have
    \begin{align*}
        h&=\frac{g_1g_2\ldots g_{k}}{\mathrm{gcd}(g_1g_2\ldots g_{k},x_{i}^{k-1}f^{k})}\\
        &=\frac{g_1g_2\ldots g_{k}}{\mathrm{gcd}(g_1g_2\ldots g_{k},x_{i}^{k-1})\mathrm{gcd}(g_1g_2\ldots g_{k},f^{k})}\\
        &=\frac{x_{i_1}\ldots x_{i_k}m}{\mathrm{gcd}(g_1g_2\ldots g_{k},x_{i}^{k-1})}.
    \end{align*}
     Since $x_{i_j}\in\mathfrak{p}$ for all $1\leq j\leq k$, $h=\frac{x_{i_1}\ldots x_{i_k}m}{\mathrm{gcd}(g_1g_2\ldots g_{k},x_{i}^{k-1})}\in\mathfrak{p}$. Therefore, $(I^{k}:g^{k-1}f)=\mathfrak{p}$, which gives $\mathrm{v}(I^{k})\leq \mathrm{v}_{\mathfrak{p}}(I^k)\leq \alpha(I)k-1$. Hence, from Lemma \ref{low}, it follows that $\mathrm v(I^{k})=\mathrm{v}_{\mathfrak{p}}(I^{k})=\alpha(I)k-1~\mbox{for all}~k\geq 1$.
\end{proof} 

The following example gives the existence of a non square-free monomial ideal, which satisfies the hypothesis of the above theorem.

\begin{example}
    Let us consider the monomial ideal $I=\langle y^2z, z^3, y^2x\rangle\subset K[x,y,z]$. Here $\Ass(I)=\{\langle x,z\rangle, \langle y,z\rangle\}$. Let us consider $f=y^2$. Then $(I:f)=\langle x,z\rangle$, which gives $\mathrm{v}(I)=2$ by Lemma \ref{low}. Since $f$ is not divisible by $x$ and $z$, it follows from the above theorem that $\mathrm{v}(I^k)=3k-1$ for all $k\geq1.$
\end{example}

\begin{proposition}\label{propv=1}
    Let $I\subset S$ be a monomial ideal with $\mathrm v(I)=1$. Then $$\mathrm v(I^{k})=2k-1~\mbox{for all}~k\geq 1.$$
\end{proposition}
\begin{proof}
   Since $\mathrm{v}(I)=1$, there exists a variable $x\in S$ such that $(I:x)=\mathfrak{p}$ for some $\mathfrak{p}\in \Ass(I).$  If $x\notin \mathfrak{p}$, then by Theorem \ref{thmpowermon}, $\mathrm v(I^{k})=2k-1$ for all $k\geq 1$. Now, suppose $x\in \mathfrak{p}$. Then $g=x^{2}\in \mathcal{G}(I)$ as $x\notin I$. Since $\mathfrak{p}\subset (I:x)$, it is clear that $\mathfrak{p}\subset(I^{k}:g^{k-1}x)$ for all $k\geq 1$. To get the reverse inclusion, we use induction on the powers of the ideal $I$. If $k=1$, then the statement follows trivially. For $k\geq 2$, let $(I^{k-1}:g^{k-2}x)\subset \mathfrak{p}$. Now, consider an element $h\in \mathcal{G}(I^{k}:g^{k-1}x)$. Then we can write 
   $$h=\frac{g_1g_2\ldots g_{k}}{\mathrm{gcd}(g_1g_2\ldots g_{k},g^{k-1}x)}=\frac{g_1g_2\ldots g_{k}}{\mathrm{gcd}(g_1g_2\ldots g_{k},x^{2k-1})},$$
   for some $g_1,\ldots,g_k\in \mathcal{G}(I)$. If any of $g_i$'s is equal to $g$, then $h\in (I^{k-1}:g^{k-2}x)$, which is contained in $\mathfrak{p}$ by induction hypothesis. Let $g_i\neq g$ for all $i\in\{1,\ldots,k\}$ and $m=\mathrm{gcd}(g_1g_2\ldots g_{k},x^{2k-1})$. Then $\deg (m)\leq k$, otherwise we get some $g_i$ such that $x^{2}|g_i$, which is a contradiction to the fact that $g_i\neq g$ for all $1\leq i\leq k$ and $g_1,\ldots,g_k,g\in\mathcal{G}(I)$. If $m=1$, then $h\in (I^{k-1}:g^{k-2}x)\subset \mathfrak{p}$. Now, for $\deg (m)\geq 1$, there exists some $g_i$ such that $\frac{g_i}{\mathrm{gcd}(g_i,x)}=\frac{g_i}{x}\mid h$. Since $\frac{g_i}{\mathrm{gcd}(g_i,x)}\in\mathfrak{p}$ for all $1\leq i\leq k$, we have $h\in\mathfrak{p}$. Therefore, $(I^{k}:g^{k-1}x)= \mathfrak{p}$, which gives $\mathrm v(I^{k})\leq 2k-1$. Using Lemma \ref{low}, we conclude that $\mathrm v(I^{k})=2k-1$ for all $k\geq 1$.
\end{proof}

The following example shows that for any monomial ideal $I$ with $\mathrm v(I)=\alpha(I)-1$, the above statement is not true.

\begin{example}\label{exampowerneqlb}
    Let $I=\langle x_1x_{2}^{2},x_{1}^{2}x_2,x_{1}^{2}x_3x_4,x_{2}^{2}x_3x_4\rangle\subset K[x_1,x_2,x_3,x_4]$ be a monomial ideal. Here $\alpha(I)=3$ and $\mathrm v(I)=2=\alpha(I)-1$, but $\mathrm v(I^2)=6\neq 2\alpha(I)-1.$
\end{example}

\begin{definition}
  The polarization of monomials of the form $x_{i}^{a_i}$, denoted by $x_{i}^{a_i}(pol)$, is defined as $x_{i}^{a_i}(pol)=\prod_{j=1}^{a_i}x_{i,j}$. For $\mathbf{X}^{\mathbf{a}}=\prod_{t=1}^{n}x_{t}^{a_t}$, $\mathbf{X}^{\mathbf{a}}(pol)=\prod_{t=1}^{n}x_{t}^{a_t}(pol)$.
If $I=\langle \mathbf{X}^{\mathbf{\alpha_1}},\ldots,\mathbf{X}^{\mathbf{\alpha_m}}\rangle\subset S$ is a monomial ideal, then $I(pol)$ is defined as follows
\begin{center}
    $I(pol)=\langle \mathbf{X}^{\mathbf{\alpha_1}}(pol),\ldots,\mathbf{X}^{\mathbf{\alpha_m}}(pol) \rangle$
\end{center}
in the ring $S(pol)=K[x_{i,j}|1\leq i \leq n, 1\leq p_i]$, where $p_i$ is the power of $x_i$ in the lcm of 
 $\{\mathbf{X}^{\mathbf{\alpha_1}},\ldots,\mathbf{X}^{\mathbf{\alpha_m}}\}$.
 \end{definition}
 
\begin{theorem}\label{thmdeg2mon}
   Let $I\subset S$ be a monomial ideal such that $\alpha(I)=2$ and $I$ has a linear resolution. Then $\mathrm{v}(I^k)=2k-1$ for all $k\geq 1$. 
\end{theorem}
\begin{proof}
    Since $I$ has a linear resolution, we have $\reg (I)=2$. It is known that $\reg(I)=\reg(I(pol))$ (\cite[Corollary 1.6.3]{Hibi}). Thus, $\reg(I(pol))=2$, which implies $I(pol)$ has a linear resolution. Since $\alpha(I)=2$, $I(pol)$ can be viewed as an edge ideal of a graph. Therefore, by using \cite[Corollary 5.3]{ficarra}, we get $\mathrm v(I(pol))=1$. Again, from \cite[Theorem 4.1.]{ficarra} it follows that $\mathrm v(I)=\mathrm v(I(pol))$. Thus, $\mathrm v(I)=1$, and hence, $\mathrm{v}(I^k)=2k-1$ for ll $k\geq 1$ by Proposition \ref{propv=1}.
\end{proof}

\begin{definition}
    Let $G$ be a simple graph. A weighted oriented graph $D$ ,whose underlying graph is $G$, is a triplet $(V(D),E(D),w)$ where $V(D)=V(G)$, $E(D)\subseteq V(D)\times V(D)$ such that $\{(x,y)\mid (x,y)\in E(D)\}$ and $w:V(D)\rightarrow \mathbb N$ is a function. The weight of $x\in V(D)$ is $w(  x).$ An ordered pair $(x,y)\in E(D)$ if there is a directed edge from the vertex $x$ to $y$. Then edge ideal of $D$ is defined as follows:
    $$I(D)=\langle x_ix_j^{w(x_j)}~|~(x_i,x_j)\in E(D)\rangle.$$
\end{definition}

\begin{theorem}\label{thmwoglp}
    Let $  {D}$ be a connected weighted oriented graph with $V({D})=\{x_1,\ldots,x_t\}$ and $I({D})\neq I(G)$, where $G$ is the underlying graph of $  {D}$. If $I(  {D})$ has linear resolution, then $\mathrm v(I({D})^k)=\alpha(I(D))k-1$ for all $k\geq 1$.
\end{theorem}
\begin{proof} 
By hypothesis, $I({D})$ has a linear resolution. Then by using \cite[Theorem 3.5.]{banerjee}, we get $I(D)^k=x_{i}^{kw}\langle x_1,\ldots,x_{i-1},x_{i+1},\ldots,x_t\rangle^k $ for all $k\geq 1$, where $i\in \{1,\dots,t\}$ and $w>1$. This gives $(I(  {D})^k:(x_{i}^{w}x_1)^{k-1}x_{i}^{w})=\langle x_1,\ldots,x_{i-1},x_{i+1},\ldots,x_t\rangle$. Thus, $\mathrm v(I(D)^k)\leq(k-1)(w+1)+w=k(w+1)-1=\alpha(I(D))k-1$. Hence, by Lemma \ref{low}, we have $\mathrm v(I(D)^k)=\alpha(I(D))k-1$ for all $k\geq 1$.
\end{proof}

\begin{definition}\label{defvertexsplit}{\rm 
A monomial ideal $I\subseteq R=K[X]$ is called {\it vertex splittable} if it can be obtained by the following recursive procedure.
\begin{enumerate}
    \item If $v$ is a monomial and $I=(v)$, $I=(0)$ or $I=R$, then $I$ is vertex splittable.
    \item If there is a variable $x$ in $R$ and vertex splittable ideals $I_1$ and $I_2$ in $K[X\setminus \{x\}]$ such that $I=xI_1+I_2$, $I_2 \subseteq I_1 $ and $\mathcal{G}(I)= \mathcal{G}(xI_1)\sqcup \mathcal{G}(I_2)$, then $I$ is a vertex splittable. For $I=xI_1+I_2$, the variable $x$ is said to be a {\it splitting variable} for $I$.
\end{enumerate}
}
\end{definition}
\begin{theorem}\label{thmvsplit}
    Let $I$ be a proper equigenerated vertex splittable ideal of $S$. Then $\mathrm{v}(I^{k})=\alpha(I)k-1$ for all $k\geq 1$.
\end{theorem}

\begin{proof}
    We proceed by induction on $\alpha(I)$. If $\alpha(I)=1$, then $I$ is a prime ideal generated by a set of variables. In this case, taking a variable $x\in I$, we get $I^{k}:x^{k-1}=I$ and the result follows. Now, let us assume that the result is true for any proper equigenerated vertex splittable ideal $I^{'}$ with $\alpha(I^{'})=\alpha(I)-1$. Now, $I$ being vertex splittable, we can write $I=xI_1+I_2$, where $x$ is a variable, $I_2\subseteq I_1$, $I_1$ and $I_2$ are vertex splittable ideals such that $G(I)=G(xI_1)\sqcup G(I_2)$ and no generators of $I_j$ is divisible by $x$ for each $j=1,2$. We may assume $\alpha(I)>1$, and thus, $I_1$ is a proper equigenerated vertex splittable ideal with $\alpha(I_1)=\alpha(I)-1$. By induction hypothesis, we have $\mathrm{v}(I_1^{k})=\alpha(I_1)k-1$. Therefore, there exist $\mathfrak{p}\in\mathrm{Ass}(I_1^{k})$ and a monomial $f^{'}_k$ with $\mathrm{deg}(f^{'}_k)=\alpha(I_1)k-1$ such that $I_{1}^{k}:f^{'}_k=\mathfrak{p}$. Let us consider the monomial $f_k=f^{'}_{k}x^{k}$. Now, $\mathfrak{p}f_{k}^{'}\subseteq I_{1}^{k}$ implies $\mathfrak{p}f_{k}\subseteq (xI_1)^{k}\subseteq I^k$, i.e., $\mathfrak{p}\subseteq I^{k}:f_k$. Let $u_1\ldots u_r u_{r+1}\ldots u_k\in G(I^{k})$ be such that $u_1,\ldots,u_r\in G(xI_1)$ and $u_{r+1},\ldots,u_{k}\in G(I_2)$. Now, we can write $u_{i}=v_{i}x$, where $v_i\in G(I_1)$ and $1\leq i\leq r$. Also, $I_2$ being contained in $I_1$, we can write $u_i=v_{i}z_{i}$ for $r+1\leq i\leq k$, where $v_i\in G(I_1)$ and $z_i$'s are some monomials. Let us consider the monomial $g=\frac{u_1\ldots u_k}{\mathrm{gcd}(u_1\ldots u_k, f_k)}$. From our choices of $u_i$'s, it is clear that $x^{r}$ divides $u_1\ldots u_k$, but $x^{r+1}$ does not divide $u_1\ldots u_k$. Thus, $g$ is equal to $\frac{v_1\ldots v_k z_{r+1}\ldots z_k}{\mathrm{gcd}(v_1\ldots v_k z_{r+1}\ldots z_k, f^{'}_{k})}$, which is divisible by $\frac{v_1\ldots v_k}{\mathrm{gcd}(v_1\ldots v_k, f^{'}_{k})}$. Now, $v_1\ldots v_k\in G(I_1^{k})$ and $I_{1}^{k}:f^{'}_{k}=\mathfrak{p}$ together give $\frac{v_1\ldots v_k}{\mathrm{gcd}(v_1\ldots v_k, f^{'}_{k})}\in \mathfrak{p}$, which again imply $g\in \mathfrak{p}$. Since $I^{k}:f_k=\big\langle \frac{u_1\ldots u_k}{\mathrm{gcd}(u_1\ldots u_k, f_k)}\mid u_1,\ldots,u_k\in G(I)\big\rangle$ and we choose $u_1,\ldots,u_k\in G(I)$ in an arbitrary way, we have $I^{k}:f_{k}=\mathfrak{p}$. Note that $\mathrm{deg}(f_k)=\mathrm{deg}(f_{k}^{'})+k=\alpha(I)k-1$. Hence, by Lemma \ref{low}, $\mathrm{v}(I^{k})=\alpha(I)k-1$ for all $k\geq 1$.
\end{proof}

\section{Stability index of $\mathrm{v}$-numbers of edge ideals}\label{secvpoweredge}

In this section, we study the stability index of $\mathrm{v}$-number for edge ideals of graphs. Also, we give the explicit expression of the $\mathrm{v}$-function.

\begin{theorem}\label{bound}
    Let $G$ be a connected graph with $m$ edges. Then $\mathrm v(I(G)^k)=2k-1$ for all $k\geq m+1$.
\end{theorem}
\begin{proof}
    Let $V(G)=\{x_1,\ldots,x_n\}$ and $I=I(G)=\langle e_1,\ldots,e_m\rangle$ be the edge ideal of a connected graph $G$, where $e_1,\ldots,e_m$ denote the monomials corresponding to the edges of $G$. Then by Theorem \ref{even}, $$(I^{m+1}:e_1\ldots e_m)=I+\langle x_ix_j~|~x_i~\mbox{is even connected to}~x_j\rangle.$$
    It is clear that for all $i\in\{2,\ldots,n\}$, $x_i$ and $x_1$ are either even-connected or odd-connected by paths in $G$. Now we get the following two cases:\\
    \textbf{Case-1:} Let $G$ be a bipartite graph. Then $\mathcal{G}(I^{m+1}:e_1\ldots e_m)$ contains only square-free elements. If not, then there exists $x_{j}^2\in\mathcal{G}(I^{m+1}:e_1\ldots e_m).$ By the definition of even-connected, we get an odd cycle in $G$, which contains $x_j$. It contradicts the fact that $G$ is bipartite. Hence,
    $$((I^{m+1}:e_1\ldots e_m):x_1)=\langle x_i\in V(G)~|~x_i~\mbox{is either adjacent or even-connected to}~x_1 \rangle,$$ 
    which gives $\mathrm v(I^{m+1})\leq 2m+1$. Since $2m+1\leq \mathrm{v}(I^{m+1})$, $\mathrm{v}(I^{m+1})=2m+1.$\\
    \textbf{Case-2:} Let $G$ be a non-bipartite graph and $x_{j}^{2}\in \mathcal{G}(I^{m+1}:e_1\ldots e_m)$ for some $j\in\{2,\ldots,n\}$. If $x_j$ is even-connected to $x_1$, then $x_j\in ((I^{m+1}:e_1\ldots e_m):x_1)$. Otherwise, there exists an odd cycle $C_r$ such that $x_i\in V(C_r)$. This implies that there exists another path, by which $x_1$ and $x_j$ are even-connected. Hence $x_1x_j\in \mathcal G(I^{m+1}:e_1\ldots e_m)$. Thus $((I^{m+1}:e_1\ldots e_m):x_1)$ is a prime ideal, which gives $\mathrm v(I^{m+1})\leq 2m+1$. Therefore $\mathrm v(I^{m+1})=2m+1$.
    \par 
    By Theorem \ref{even}, it is clear that $(I^{m+1}:e_1\ldots e_m)$ contains all edges and also even connected edges of $G$. Hence $(I^{m+1+s}:e_{1}^{s+1}e_2\ldots e_m)=(I^{m+1}:e_1\ldots e_m)$ for all $s\geq 0$. This gives that $\mathrm v(I^{m+1+s})\leq 2(m+s+1)-1$ for all $s\geq 0.$ Hence, by Lemma \ref{low}, $\mathrm v(I^{m+s+1})=2(m+s+1)-1$. We conclude that $\mathrm v(I^{k})=2k-1$ for all $k\geq m+1$.
\end{proof}

However, the following example shows that the above theorem is not true for any graph. 

\begin{example}\label{examdiscon}
    Let $I=\langle x_1x_2,x_2x_3,x_4x_5,x_5x_6\rangle\subset K[x_1,\ldots,x_6]$ be the edge ideal of a disconnected graph $G$. Here, $\Ass(I)=\{\langle x_2,x_5\rangle,\langle x_1,x_3,x_5\rangle,\langle x_2,x_4,x_6\rangle,\langle x_1,x_3,x_4,x_6\rangle\}$, and $\mathrm v(I)=2$. It is clear that $(I:x_1x_6)=\langle x_2,x_5\rangle$. Since $G$ is a bipartite graph, $\Ass(I^k)=\Ass(I)$ for all $k\geq 1.$ This gives $(I^k:(x_1x_2)^{k-1}x_1x_6)=\langle x_2,x_5\rangle.$ Hence $\mathrm v(I^k)\leq 2k,$ for all $k\geq 1.$ If possible let for some $k$, $\mathrm v(I^k)=2k-1.$ This implies that there exists a monomial $f$ such that $(I^k:f)=\mathfrak{p}$, where $\mathfrak{p}\in\Ass(I)$ and $\deg(f)=2k-1$. Then there exists two different variables $x_i$ and $x_j$ belong to $\mathfrak{p}$ such that $$x_i=\frac{(x_1x_2)^{k_1}(x_2x_3)^{k_2}(x_4x_5)^{k_3}(x_5x_6)^{k_4}}{\mathrm{gcd}((x_1x_2)^{k_1}(x_2x_3)^{k_2}(x_4x_5)^{k_3}(x_5x_6)^{k_4},f)}$$
    and
    $$x_j=\frac{(x_1x_2)^{k_{1}^{'}}
    (x_2x_3)^{k_{2}^{'}}(x_4x_5)^{k_{3}^{'}}(x_5x_6)^{k_{4}^{'}}}{\mathrm{gcd}((x_1x_2)^{k_{1}^{'}}(x_2x_3)^{k_{2}^{'}}(x_4x_5)^{k_{3}^{'}}(x_5x_6)^{k_{4}^{'}},f)},$$ 
    where $k_1+k_2+k_3+k_4=k_{1}^{'}+k_{2}^{'}+k_{3}^{'}+k_{4}^{'}=k$, $x_i\in\{x_1,x_2,x_3\}$ and $x_j\in\{x_4,x_5,x_6\}$. This gives that $f=\displaystyle\frac{(x_1x_2)^{k_1}(x_2x_3)^{k_2}(x_4x_5)^{k_3}(x_5x_6)^{k_4}}{x_i}=\displaystyle\frac{(x_1x_2)^{k_{1}^{'}}
    (x_2x_3)^{k_{2}^{'}}(x_4x_5)^{k_{3}^{'}}(x_5x_6)^{k_{4}^{'}}}{x_j}$. By comparing powers of variables, we get $2(k_1+k_2)-1=2(k_{1}^{'}+k_{2}^{'})$ and $2(k_3+k_4)-1=2(k_{3}^{'}+k_{4}^{'})-1$, which is not possible. Hence, $\mathrm{v}(I^k)=2k$ for all $k\geq 1.$  
\end{example}

    The next result says that the regularity of any large power of edge ideal of any graph is always greater than or equal to its $\mathrm v$-number. 
    
\begin{corollary}
    Let $G$ be a graph with $m$ edges. Then $\mathrm v(I^k)\leq \reg(I^k)$ for all $k\geq m+1$.
\end{corollary}
\begin{proof}
Consider $I=I(G)$. From \cite[Theorem 4.7]{tai}, we get $\reg(I^k)\geq 2k+\inma (G)-1$ for all $k\geq 1$. It gives $\reg(I^k)\geq 2k-1$. Now, by using Theorem \ref{bound}, $\mathrm v(I^k)\leq \reg(I^k)$ for all $k\geq m+1$.
\end{proof}

However, the above statement is not true for any square-free monomial ideal.
\begin{example}\label{examsqfreenottrue}
    Let $I=\langle x_1x_2x_3,x_3x_4x_5,x_5x_6x_7\rangle\subset K[x_1,\ldots,x_7]$ be a square-free monomial ideal. Then it is easy to verify that
    \begin{align*}
        \Ass(I)=\{&\langle x_1,x_5\rangle,\langle x_2,x_5\rangle,\langle x_3,x_5\rangle,\langle x_3,x_6\rangle,\langle x_3,x_7\rangle, \langle x_1,x_4,x_6\rangle,\\
        &\langle x_1,x_4,x_7\rangle,\langle x_2,x_4,x_6\rangle,\langle x_2,x_4,x_7\rangle\}.
    \end{align*}
    If possible, let $\mathrm v(I^k)=3k-1$ for some $k\geq 1$. Then there exists $\mathfrak{q}\in \Ass(I^k)$ such that $(I^k:f)=\mathfrak{q}$ for some monomial $f$ with $\deg(f)=3k-1$. Now, it is clear that $\mathfrak{q}\supset \mathfrak{p}$ for some $\mathfrak{p}\in \Ass(I)$. We can choose two different variables $x_i,x_j\in \mathfrak{p}$ such that $x_i\in\{x_1,x_2,x_3\}$ and $x_j\in\{x_5,x_6,x_7\}$. Then $x_i=\displaystyle\frac{u}{\mathrm{gcd}(u,f)}$ and $x_j=\displaystyle\frac{u'}{\mathrm{gcd}(u',f)}$ for some $u,u'\in\mathcal{G}(I^k)$. Now, we can write $u=(x_1x_2x_3)^{k_1}(x_3x_4x_5)^{k_2}(x_5x_6x_7)^{k_3}$ and $u'=(x_1x_2x_3)^{k_{1}^{'}}(x_3x_4x_5)^{k_{2}^{'}}(x_5x_6x_7)^{k_{3}^{'}}$ for some $k_1,k_2,k_3,k_{1}^{'},k_{2}^{'},k_{3}^{'}\in \mathbb{Z}_{\geq 0}$ with $k_1+k_2+k_3=k$ and $k_{1}^{'}+k_{1}^{'}+k_{3}^{'}=k$. Since $\deg(f)=3k-1$ and $\deg(u)=\deg(u')=3k$, we have $f=\displaystyle\frac{u}{x_i}=\displaystyle\frac{u'}{x_j}$. Without loss of generality, let us consider $x_i=x_1$. Then comparing degree of $x_2$ in $\frac{u}{x_1}$ and $\frac{u'}{x_j}$, we get $k_1=k_{1}^{'}$, but comparing degree of $x_1$ in $\frac{u}{x_1}$ and $\frac{u'}{x_j}$, we get $k_1-1=k_{1}^{'}$. This gives a contradiction. Hence, $\mathrm v(I^{k})\neq 3k-1$ for all $k\geq 1.$
\end{example}

Consider $d(x,y)=\mbox{length of a shortest path between two vertices}~x~\mbox{and}~y$. In the following theorem, we give a better bound of $\text{v-stab}(I(G))$ for bipartite graphs.

\begin{theorem}\label{thmvstabbipartite}
    Let $G$ be a connected bipartite graph with partite sets $X$ and $Y$. Then $\mathrm v(I^k)=2k-1$ for all $k\geq \min\{|Y|+1-\max\{\deg (x)~|~x\in X\},|X|+1-\max\{\deg (y)~|~y\in Y\}\}$.
\end{theorem}
\begin{proof}
Let $V(G)=\{x_1,x_2,\ldots,x_m,y_1,y_2,\ldots,y_n\}$ with $X=\{x_1,x_2,\ldots,x_m\}$ and $Y=\{y_1,y_2,\ldots,y_n\}$. Since $G$ is bipartite, for any $x_i\in X$, $\max\{d(x_i,y_j)~|~y_j\in Y\}$ is an odd positive integer, say $2t+1$ for some $t\in\mathbb{Z}_{\geq 0}$. Consider $Y_s=\{y_j\in Y~|~d(x_i,y_j)=2s+1\}$ and $a_s=|Y_s|$. Now for $s=0$, there exists some $y_j\in Y$ such that $x_i$ and $y_j$ are adjacent. If $s=1$, then for each $y_j\in Y_1$, there exists at least one edge $e$ such that $x_iy_j\in (I^2:e)$. For each $y_j\in Y_1$, fix one $e$ satisfying the above condition. Let $A_1$ denote the set of all such $e$ for every $y_j\in Y_1.$ Note that $|A_1|\leq a_1$. Now if $s=2$, then for some $y_j\in Y$, there exists two edges $e$ and $f$ such that $e\cap f=\emptyset$ , $V(P_6)=\{x_i\}\cup e\cup f\cup \{y_j\} $ and $x_iy_j\in (I^3:ef)$, where $e\in A_1.$ For each $y_j\in Y_2$, choose only one $f$ such that $x_iy_j\in (I^3:ef)$ and let $A_2$ denote the collection of all such $f$ for every $y_j\in Y_2.$ Clearly, $|A_2|\leq a_2$. Similarly, $d(x_i,y_j)=2t+1$ gives $A_t$ and $|A_t|\leq a_t$.  It is clear that $((I^{b_i+1}:\displaystyle\prod_{s=1}^{t}\prod_{e\in A_s}e):x_i)=\langle y_j~|~y_j\in Y\rangle$, where $b_i=|\displaystyle\bigcup_{s=1}^{t}A_s|.$
    Therefore $\mathrm v((I^{b_i+1}:\displaystyle\prod_{s=1}^{t}\prod_{e\in A_s}e))=1$. We have $b_i+1=|\displaystyle\bigcup_{s=1}^{t}A_s|+1\leq \displaystyle\sum_{s=1}^{t}a_s+1=n-\deg x_i+1$ for any $x_i$, which implies that $\mathrm{v}(I^{b+1})=2b+1$, where $b=\min\{n-\deg x_i+1~|~x_i\in X\}=n+1-\max\{\deg x~|~x\in X\}$. This gives that $\mathrm v(I^k)=2k-1$ for $k\geq\displaystyle\sum_{s=1}^{t}a_s+1=n-\deg x_i+1$. This implies that $k\geq n+1-\max\{\deg x_i~|~x_i\in X\}$. Similar manner, we get $k\geq m+1-\max\{\deg y_j~|~y_j\in Y\}.$ Thus $\mathrm v(I^k)=2k-1$ for all $k\geq \min\{|Y|+1-\max\{\deg x~|~x\in X\},|X|+1-\max\{\deg y~|~y\in Y\}\}$.
\end{proof} 
By using the above theorem, we get the following corollaries about path and even cycle.
\begin{corollary}\label{corvstabpath}
    Let $P_n$ be a path and $I$ be its edge ideal. Then $\mathrm v (I^{k})=2k-1$ for all $k\geq [\frac{n}{2}]-1.$
\end{corollary}
\begin{proof}
    Using the same notations of the above one, either $|X|=[\frac{n}{2}]$ or, $|Y|=[\frac{n}{2}]$ and drgree of each of vertex is $2$. Without loss of generality, we consider $|X|=[\frac{n}{2}]$. Then it is clear that $|Y|=n-[\frac{n}{2}].$ By applying the previous statement, $k\geq \min\{|X|+1-2,|Y|+1-2\}=[\frac{n}{2}]-1.$
\end{proof}
The following example shows that the bound is sharp.
\begin{example}
    Let $I=\langle x_1x_2,x_2x_3,x_3x_4,x_4x_5,x_5x_6,x_6x_7,x_7x_8\rangle\subset K[x_1,\ldots,x_8]$ be the edge ideal of $P_8.$ Here $\mathrm{v}(I)=2$, $\mathrm{v}(I^2)=4$ and $\mathrm{v(I^k)}=2k-1$ for all $k\geq 3=[\displaystyle\frac{8}{2}]-1.$
\end{example}
\begin{corollary}\label{corvstabevencycle}
    Let $C_{n}$ be an even cycle and $I$ be its edge ideal. Then $\mathrm v(I^{k})=2k-1$ for all $k\geq [\frac{n}{2}]-1.$
\end{corollary}
\begin{proof}
    From the given hypothesis, $|X|=|Y|=\frac{n}{2}$ and degree of each vertex is $2$. Then $k\geq \min\{|X|+1-2,|Y|+1-2\}=\frac{n}{2}-1=[\frac{n}{2}]-1.$
\end{proof}
The following example shows that the bound is sharp.
\begin{example}
    Let $I=\langle x_1x_2,x_2x_3,x_3x_4,x_4x_5,x_5x_6,x_6x_7,x_7x_8,x_8x_1\rangle\subset K[x_1,\ldots,x_8]$ be the edge ideal of $P_8.$ Here $\mathrm{v}(I)=2$, $\mathrm{v}(I^2)=4$ and $\mathrm{v(I^k)}=2k-1$ for all $k\geq 3=[\displaystyle\frac{8}{2}]-1.$
\end{example}
For odd cycles, the following example shows that the above statement does not hold.
\begin{example}
    Let $I$ be the edge ideal of $C_7.$ Then $\mathrm v(I)=2$, $\mathrm v(I^2)=4$ and $\mathrm v(I^{k})=2k-1$ for all $k\geq 3$. Here the stage of stabilization is $3>[\frac{7}{2}]-1.$
\end{example}

In the following proposition, we give a stage of stabilization for odd cycles.

\begin{proposition}\label{propvstaboddcycle}
    Let $C_{n}$ be an odd cycle, and $I$ be its edge ideal. Then $\mathrm v(I^{k})=2k-1$ for all $k\geq [\frac{n}{2}].$
\end{proposition}
\begin{proof}
    Let $n=2m+1$ for $m\in\mathbb N$ and $V(C_n)=\{x_1,x_2,\ldots,x_n\}$. Consider $J=(I^m:x_1x_2\ldots x_{2m-2})$. By the construction of $J$, it is clear that $x_{2m+1}x_3,x_{2m+1}x_5,\ldots,x_{2m+1}x_{2m-1}\in J.$ 
    Hence $(J:x_{2m+1})=\{x_1,x_3,\ldots,x_{2m-1},x_{2m}\}$. Therefore $\mathrm v(I^m)=2m-1$. Since $J=(I^m:x_1x_2\ldots x_{2m-2})$ and $(I^m:x_1x_2\ldots x_{2m-2})=(I^{m+l}:(x_1x_2)^{l+1}x_3x_4\ldots x_{2m-2})$ for $l\geq 0$, $\mathrm v(I^{m+l})=2(m+l)-1.$ This proves our statement.
\end{proof}

\end{document}